\documentclass[11pt]{article} 

\usepackage{amsmath, amssymb, amsthm}
\usepackage{color}
\usepackage[dvipdfm,a4paper,left=30mm,right=30mm,top=35mm,bottom=35mm]{geometry}
\usepackage[titletoc, toc, title]{appendix}

\begin{document}
\renewenvironment{proof}{ \noindent {\bfseries Proof.}}{qed}
\newtheorem{defin}{Definition}
\newtheorem{theorem}{Theorem}
\newtheorem{remark}{Remark}
\newtheorem{proposition}{Proposition}
\newtheorem{lemma}{Lemma}
\newtheorem{cor}{Corollary}
\def\begproof{\noindent{\bf Proof: }}
\def\endproof{\quad\vrule height4pt width4pt depth0pt \medskip}
\def\div{\nabla\cdot}
\def\rot{\nabla\times}
\def\sign{{\rm sign}}
\def\arsinh{{\rm arsinh}}
\def\arcosh{{\rm arcosh}}
\def\diag{{\rm diag}}
\def\const{{\rm const}}
\def\eps{\varepsilon}
\def\phi{\varphi}
\def\theta{\vartheta}
\newcommand{\Bchi}{\mbox{$\hspace{0.12em}\shortmid\hspace{-0.62em}\chi$}}
\def\C{\hbox{\rlap{\kern.24em\raise.1ex\hbox
      {\vrule height1.3ex width.9pt}}C}}
\def\R{\mathbb{R}}
\def\P{\hbox{\rlap{I}\kern.16em P}}
\def\Q{\hbox{\rlap{\kern.24em\raise.1ex\hbox
      {\vrule height1.3ex width.9pt}}Q}}
\def\M{\hbox{\rlap{I}\kern.16em\rlap{I}M}}
\def\N{\hbox{\rlap{I}\kern.16em\rlap{I}N}}
\def\Z{\hbox{\rlap{Z}\kern.20em Z}}
\def\K{\mathcal{K}}
\def\({\begin{eqnarray}}
\def\){\end{eqnarray}}
\def\[{\begin{eqnarray*}}
\def\]{\end{eqnarray*}}
\def\part#1#2{{\partial #1\over\partial #2}}
\def\partk#1#2#3{{\partial^#3 #1\over\partial #2^#3}}
\def\mat#1{{D #1\over Dt}}
\def\dt{\, \partial_t}
\def\as{_*}
\def\d{\, {\rm d}}
\def\e{ \text{e}}
\def\D{\mathcal{D}}
\def\feps{f_\eps}
\def\pmb#1{\setbox0=\hbox{$#1$}
  \kern-.025em\copy0\kern-\wd0
  \kern-.05em\copy0\kern-\wd0
  \kern-.025em\raise.0433em\box0 }
\def\bar{\overline}
\def\lbar{\underline}
\def\fref#1{(\ref{#1})}

\begin{center}
{\LARGE Fractional-diffusion-advection limit of a kinetic model}
\bigskip

{\large P. Aceves-S\'anchez}\footnote{Fakult\"at f\"ur Mathematik, Universit\"at Wien.} and
{\large C. Schmeiser}$^1$
\end{center}
\vskip 1cm

\noindent{\bf Abstract.} A fractional diffusion equation with advection term is rigorously derived from a kinetic transport model
with a linear turning operator, featuring a fat-tailed equilibrium distribution and a small directional bias due to a given vector field.
The analysis is based on bounds derived by relative entropy inequalities and on two recently developed approaches for the macroscopic
limit: a Fourier-Laplace transform method for spatially homogeneous data and the so called moment method, based on a modified
test function.

\vskip 1cm

\medskip

\vskip 1cm

\noindent{\bf Acknowledgment:} This work has been supported by the PhD program {\em Dissipation and Dispersion
in Nonlinear PDEs} funded by the Austrian Science Fund, grant no. W1245. P.A.S. acknowledges support from Consejo Nacional
de  Ciencia y Tecnologia and from the {\em beca complemento} program of Secreteria de Educacion Publica, Mexico. The 
work of C.S. has also been supported by the Vienna Science and Technology Fund, grant no.  LS13-029.




\section{Introduction}

The goal of this paper is to study the limit as $\eps\to 0$ of the distribution function $\feps(x,v,t)$ (depending on position 
$x\in\R^N$, velocity $v\in\R^N$, and time $t\ge 0$), solving the kinetic Cauchy problem
\begin{equation}\label{kineticeq}
\begin{array}{rcll}
\eps^\alpha \partial_t \feps + \eps v \cdot \nabla_x \feps & = & Q_\eps (\feps) & \text{in } \R^N \times \R^N\times \R^+ \,,\\
\feps (t = 0)  & = & f^{in}             & \text{in } \R^N \times \R^N \,, 
\end{array} 
\end{equation}
where the linear collision operator is given by
\(
  &&Q_\eps(f) = \int_{\R^N} (T_\eps(v'\to v,x,t)f' - T_\eps(v\to v',x,t)f)dv' \,,\label{Qeps}\\
  &&\mbox{with } T_\eps(v'\to v,x,t) = \left(1+ \eps^{\alpha-1}\Phi(v,v',c(x,t))\right) M(v) \,,\nonumber
\)
with the prescribed vector field $c(x,t)\in\R^N$, and with the equilibrium distribution $M(v)$ with the properties
\begin{eqnarray}
 && M>0 \,,\quad M \mbox{ is rotationally symmetric,} \quad \int_{\R^N} M(v)dv = 1 \,,\nonumber\\
 && M(v) = \frac{\gamma}{| v|^{N + \alpha}} \,,\qquad\mbox{for } |v| \geq 1\,,\quad 1<\alpha<2\,,\ \gamma>0 \,.\label{eqdecay}
\end{eqnarray}
The decay property is responsible for the choice of the scaling in \eqref{kineticeq}, which will turn out to be significant in the following.
Note that $M$ has finite first order but not second order moments. As usual in kinetic theory, $f'$ denotes evaluation at $v'$. 

The collision operator can be written as $Q_\eps = Q_0 + \eps^{\alpha-1}Q_1$ with the dominating, directionally unbiased
relaxation operator 
\[
   Q_0(f) = \rho_f M - f \,,\qquad \rho_f := \int f\,dv \,,
\]
and the turning operator
\[
  Q_1(f) = \int [\Phi(v,v',c)M f' - \Phi(v',v,c)M'f]dv' \,,
\]
supposed to bias velocity changes towards the direction given by $c$. Here and in the following, $\d v$, $\d v'$, and $\d x$
denote the Lebesgue measure on $\R^N$ and $\d t$ the Lebesgue measure on $\R^+$, which always have to be understood as the integration domains, except stated otherwise. In the scaling process, the ratio of characteristic times
between the biased and unbiased velocity jump mechanisms has been denoted by $\eps^{\alpha-1}$, 
and then macroscopic length and time scales have been introduced.

A possible motivation for the model \eqref{kineticeq} is the description of ensembles of motile microorganisms, subject to a
chemical signal encoded in the vector field $c$, which might be interpreted as the spatial gradient of a chemo-attractant.
One of the best studied microorganisms is the bacterium \emph{Escherichia coli}, whose 
swimming pattern can be described as a run-and-tumble process \cite{berg2004coli,Berg1972}, characterized by periods of straight running alternated with (much shorter) periods of reorientation (or tumbling). Under the idealizing assumption of instantaneous
velocity jumps, this can be described stochastically by kinetic transport equations, which have been introduced as models for
microorganisms in the pioneering works \cite{alt1980biased} and \cite{othmer1988models}.
In the presence of a chemo-attractant gradient, the velocity jump process is biased, which is described by the function 
$\Phi(v,v^\prime,c)$, depending on the velocities $v^\prime$ and $v$ before and, respectively, after the jump, and on the gradient $c$. 

From a macroscopic point of view (where length and time scales are large compared to individual runs), a standard description of the resulting
motility is by Brownian motion with a drift. On the other hand, the recent progress in tracking individual trajectories
\cite{barkai2012strange,bartumeus2003helical,escudero2006fractional,klafter1990microzooplankton,metzler2000random} allowed to show that the movement of certain microorganisms is better described by a so-called L\'{e}vy flight. In particular, there is evidence that  
\emph{E. coli} adopts a L\'evy flight type movement when there is scarcity of food resources \cite{Wu2006}.
Macroscopically, L\'evy flights show a scaling behavior different from Brownian motion, where the average displacement scales with
the square root of time, a behavior called fractional Brownian motion.
In the model considered here, this kind of behavior is described by a high probability of larger velocities encoded in the fat tail 
\eqref{eqdecay} of the equilibrium distribution $M$.

For an equilibrium distribution $M$ with finite second order moments, the scaling with $\alpha=2$ would be appropriate, and the
macroscopic limit $\eps\to 0$ would lead to a convection diffusion equation for the limit of the macroscopic density $\rho_f$
(see, e.g., \cite{MR2065025}). On the other hand it has been shown in \cite{ben2011anomalous,MR2815035,MR2763032} that with the assumption \fref{eqdecay} and with $Q_1=0$ the macroscopic limit leads to a fractional diffusion equation; see also \cite{Jara2009},  
where this has been carried out via a probabilistic approach. 
Fractional diffusion equations with advection have been the object of intense study in recent years. Issues such as regularity have been
addressed by many authors, most notably by Silvestre and co-workers \cite{silvestre2010differentiability,SilVicolZlatos}. The problem of a rigorous derivation of a fractional diffusion equation with convection from kinetic models has been posed, but left open in \cite{MR2815035}. This is the purpose of the present work.

Two different methods will be used, leading to results with slightly different assumptions on the data. The Laplace-Fourier transform
approach of \cite{MR2763032} can only be used in the case of constant $c$. On the other hand, it requires milder assumptions on 
the turning rate $\Phi$ than the moment method of \cite{MR2815035}. In these works, the coercivity properties of the leading order
collision operator $Q_0$ are the essential ingredient for obtaining estimates uniform with respect to $\eps$. The important 
contribution of the present work is to employ the equilibrium distribution of the full collision operator $Q_\eps$ and a corresponding
entropy dissipation property. The latter holds although detailed balance is not required, as has first been shown in \cite{MR1803225}
and actually is known now as a general result for generators of Markov processes \cite{FontbonaJourdain}, like $Q_\eps$,
which is obviously preserving positivity and conserving mass:
$$
\int Q_\varepsilon (f) \d v = 0 \,.
$$
Fractional diffusion is generated by a fractional power of the Laplacian (\emph{fractional Laplacian}), which can be defined via the 
Fourier transform $\mathcal{ F}$ as a multiplication operator in Fourier coordinates,
\begin{equation}\label{fracLaplacian}
\mathcal{ F}( ( - \Delta)^{\alpha /2} \rho) ( k) : =  |k|^\alpha \mathcal{ F} (\rho) ( k) \,,
\end{equation}
or as a singular integral,
\begin{equation}\label{fracLaplacian2}
( - \Delta)^{ \alpha / 2} \rho( x) = c_{ N, \alpha} \, \, \text{P.V.} \int_{ \R^N} \frac{ \rho( x) - \rho( y)}{ | x - y|^{ N + \alpha}} \d y \,, 
\end{equation}
where P.V. denotes the Cauchy principal value, and 
$$
c_{ N, \alpha} = 
  \Gamma(\alpha+1) \left( \int_{\R^N} \frac{w_1^2 |w|^{-N-\alpha}}{1+w_1^2} \d w\right)^{-1}\,,
$$
with the Gamma function $\Gamma$. The value of $c_{N,\alpha}$ will be verified by our results below. 
Note that for $\alpha>1$ a principal value can be avoided by the equivalent representation
\begin{equation*}
  (-\Delta)^{\alpha/2} \rho(x) = c_{ N, \alpha}  \int_{ \R^N} \frac{\rho(x) - \rho(y) - (x-y)\cdot\nabla_x \rho(x)}{| x - y|^{ N + \alpha}}\d y \,.
\end{equation*}
For a detailed discussion of the properties of the fractional Laplacian consult 
\cite{MR0350027,DiNezza2012521,stein1970singular}. We only note that it is formally self-adjoint, which is a straightforward 
consequence of both representations \eqref{fracLaplacian} and \eqref{fracLaplacian2}.

The main result of this work is the rigorous validity of the macroscopic limit $\eps\to 0$:

\begin{theorem}\label{Theorem1}
Let $f^{in}\in L^2(\d v\d x/M)$, $(1+|v|)f^{in} \in L_+^1(\d v\,\d x)$, and either
\begin{eqnarray*}
  \mbox{\bf Assumption A: } && c={\rm const}\in\R^N \,,\quad (1+|v| + |v^\prime|)\Phi \mbox{ is bounded, or}\\
  \mbox{\bf Assumption B: } && c={\rm const}\in\R^N \,,\quad \Phi \mbox{ is bounded,} \quad \int \Phi(v',v,c)M'dv'=0 \,, 
      \mbox{ or} \\
  \mbox{\bf Assumption C: } && c\in W^{1,\infty}(\d x\d t)^N \,,\quad (1+|v| + |v^\prime|)\Phi \mbox{ is bounded} \\ 
  && \mbox{and Lipschitz continuous with respect to } c\,,
\end{eqnarray*}
Then there exists 
$\rho\in L_{loc}^\infty(\d t; L^2(\d x))$, such that the solution $f_\eps$ of \fref{kineticeq} converges, as $\eps\to 0$, to
$\rho M$ in $L_{loc}^\infty(\d t; L^2(dv\,dx/M))$ weak*, and $\rho$ solves in the distributional sense the Cauchy
problem
\(\label{macro}
\begin{array}{l}
  \partial_t\rho + \nabla_x\cdot (\rho u(c)) + A (-\Delta)^{\alpha/2} \rho = 0 \,,\\
   \rho(t=0) = \rho^{in} := \int f^{in}dv \,.
\end{array}
\)
with
$$
  A = \gamma\int_{\R^N} \frac{w_1^2 |w|^{-N-\alpha}}{1+w_1^2} \d w \,,\qquad u(c) = \int Q_1(M) v \d v \,.
$$
\end{theorem}
The main parts of the proof will be given in Sections 4 (Assumptions A and B) and 5 (Assumption C), after presentation of the formal
macroscopic limit for a simple model problem in Section 2, and the derivation of several uniform (in the small parameter $\eps$) 
bounds on the solution of \eqref{kineticeq} in Section 3.




\section{Formal asymptotics of a simple model}\label{sec:formal}

In this section the Cauchy problem \eqref{kineticeq}, \eqref{Qeps} is considered with constant $c\in\R^N$ and with the turning kernel
\begin{equation}\label{collsimple}
 \Phi(v,v^\prime,c) =  c \cdot \frac{ v }{ | v|} \,. 
\end{equation}
This means that the turning rate is independent of the incoming velocity $v^\prime$ and prefers outgoing velocities $v$ 
in the direction of $c$. The Fourier-Laplace approach to the macroscopic limit will be carried out formally, deferring a rigorous 
justification for the general form of the turning kernel satisfying Assumption A or B of Theorem \ref{Theorem1} to Section 6. 
Note that \eqref{collsimple} does not satisfy Assumption A, but Assumption B by its oddness with respect to the first variable. 

We introduce the Fourier transformation $\mathcal{F}$ with respect to $x$ and the Laplace transformation $\mathcal{L}$ with respect 
to $t$,
$$
  (\mathcal{F} f)(k) := \int \e^{ -i k \cdot x} f(x) \d x \,,\qquad
  (\mathcal{L} f)(p) := \int_{0}^{\infty} \e^{-pt } f(t) \d t \,,\qquad  p > 0 \,, \quad k \in \R^N \,,
$$
and define the Fourier-Laplace transform
\begin{equation*}
\widehat{f_\varepsilon} := \mathcal{L}\mathcal{F}f_\varepsilon  \,.
\end{equation*}
Taking the Fourier-Laplace transform of \eqref{kineticeq} with the turning kernel \eqref{collsimple} yields
\begin{align}
\varepsilon^\alpha p \widehat{ f_{\varepsilon}} - \varepsilon^{\alpha} \mathcal{F}f^{in} + \varepsilon i v \cdot k \widehat{f_{\varepsilon}}  =  M \left( 1 + \varepsilon^{\alpha - 1} c\cdot v/|v|\right) \widehat{\rho_{\varepsilon}} - \widehat{ f_{\varepsilon}} \,, \label{kinetic3} 
\end{align}
with $\rho_\eps := \rho_{f_\eps}$, where the evenness of $M$ has been used. This can be rewritten as
\begin{equation}
\widehat{ f_\varepsilon} = \frac{\varepsilon^\alpha \mathcal{F} f^{in}}{1 + \varepsilon^\alpha p + \varepsilon i v \cdot k} + \frac{M \left( 1 + \varepsilon^{\alpha - 1} c\cdot v/|v|\right) \widehat{\rho_{\varepsilon}}}{1 + \varepsilon^\alpha p + \varepsilon i v \cdot k}  \,.
\end{equation}
Integration with respect to $v$ leads to a closed equation for $\widehat{\rho_{\varepsilon}}$ (a consequence of the simple form of the
model problem), which can be written in the form
\[
  &&\Biggl( \int \frac{\eps^\alpha p + \eps^{2\alpha}p^2 + \eps^2 (v\cdot k)^2 + \eps i v\cdot k}
  {(1+\eps^\alpha p)^2 + \eps^2 (v\cdot k)^2}M\d v \\
  &&  - \int \frac{\eps^{\alpha-1}c\cdot v/|v|(1+\eps^\alpha p - \eps i v\cdot k)}
  {(1+\eps^\alpha p)^2 + \eps^2 (v\cdot k)^2}M\d v \Biggr)\widehat{\rho_{\varepsilon}} 
  = \int \frac{\varepsilon^\alpha \mathcal{F} f^{in}}{1 + \varepsilon^\alpha p + \varepsilon i v \cdot k} \d v
\]
Again by the evenness of $M$, the imaginary part of the first integral and the real part of the second integral on the left hand side
vanish. Therefore, division by $\eps^\alpha$ gives
\(
  &&\left( \int \frac{p + \eps^{\alpha}p^2 + \eps^{2-\alpha} (v\cdot k)^2} {(1+\eps^\alpha p)^2 + \eps^2 (v\cdot k)^2}M\d v 
   + i\int \frac{c\cdot v/|v|(v\cdot k)} {(1+\eps^\alpha p)^2 + \eps^2 (v\cdot k)^2}M\d v \right)\widehat{\rho_{\varepsilon}}
       \nonumber\\
  && = \int \frac{\mathcal{F} f^{in}}{1 + \varepsilon^\alpha p + \varepsilon i v \cdot k} \d v \label{rho_eps}
\)
Now we are prepared for formally passing to the limit, which is easy for all terms except one (because of the nonexistence of second
order moments of $M$):
$$
   \int \frac{\eps^{2-\alpha} (v\cdot k)^2}
  {(1+\eps^\alpha p)^2 + \eps^2 (v\cdot k)^2}M\d v = O(\eps^{2-\alpha}) + \eps^{2-\alpha} \gamma\int_{|v|>1} 
  \frac{ (v\cdot k)^2 |v|^{-N-\alpha}}{(1+\eps^\alpha p)^2 + \eps^2 (v\cdot k)^2}\d v
$$
With the coordinate transformation $v = (w_1 k/|k| + w^\bot)/(\eps |k|)$ (a stretching and a rotation), it becomes clear that
the right hand side converges as $\eps\to 0$ to 
$$
  A |k|^\alpha \,,\qquad\mbox{with } A = \gamma\int_{\R^N} \frac{w_1^2 |w|^{-N-\alpha}}{1+w_1^2} \d w > 0 \,.
$$
For the computation of the limit of the imaginary term in \eqref{rho_eps}, the rotational symmetry of $M$ is used. For the 
formal limit $\rho$ of $\rho_\eps$ we obtain
$$
  \left(p + A|k|^\alpha + i B c\cdot k \right)\widehat{\rho} = \mathcal{F}\rho^{in} = \int \mathcal{F}f^{in} \d v \,,\qquad
  \mbox{with } B = \frac{1}{N} \int |v| M(v)\d v > 0 \,.
$$
This is the Fourier-Laplace transformed version of the Cauchy problem
$$
  \partial_t \rho + \nabla_x\cdot (\rho Bc) + A(-\Delta)^{\alpha/2}\rho = 0 \,,\qquad \rho(t=0) = \rho^{in} \,,
$$
i.e. \eqref{macro} with 
$$
  u(c) = Bc = \int Q_1(M)v\,\d v \,,\qquad Q_1(M) = c\cdot \frac{v}{|v|}M \,.
$$

\section{Uniform bounds}

The derivation of bounds uniform in the small parameter $\eps$ is based on the equilibrium distribution $F_\eps(v;x,t)$ of the full
collision operator $Q_\eps = Q_0 + \eps^{\alpha-1}Q_1$, defined as solution of the problem
\(\label{F-eps}
  Q_\eps(F_\eps) = 0 \,,\qquad \int F_\eps \d v = 1 \,.
\)
In this problem, $x$ and $t$ play the role of parameters, present through the dependence of $Q_1$ on the vector field $c(x,t)$.

\begin{lemma}
Let the assumptions of Theorem \ref{Theorem1} hold. Then, for $\eps>0$ small enough, the problem \eqref{F-eps} has a unique
solution $F_\eps$ satisfying
\(\label{F-eps-est}
  \frac{1-\eps^{\alpha-1}\overline\Phi}{1+\eps^{\alpha-1}\overline\Phi} \le \frac{F_\eps}{M} \le 
  \frac{1+\eps^{\alpha-1}\overline\Phi}{1-\eps^{\alpha-1}\overline\Phi} \,,
\)
where $\overline\Phi$ is an upper bound for the modulus $|\Phi|$ of the turning kernel. Furthermore, 
\(\label{F-eps-deriv}
  \left| \frac{\partial_t F_\eps}{F_\eps} \right| , \, \left| \frac{v\cdot\nabla_x F_\eps}{F_\eps} \right| \le \eps^{\alpha-1} \lambda \,,
\)
with the constant $\lambda$ independent of $\eps$, and $\lambda=0$ under Assumptions A and B.
\end{lemma}

\begin{proof}
The existence and uniqueness result follows from the appendix of \cite{MR1803225} or from \cite{MR0423039}, but it can also be
easily derived by contraction, using the fixed point formulation for $G_\eps = \frac{F_\eps}{M} \in L^\infty(\d v)$.
$$
   G_\eps(v) = \frac{1+\eps^{\alpha-1}\int \Phi(v,v^\prime,c)M^\prime G_\eps^\prime \d v^\prime}
     {1+\eps^{\alpha-1}\int \Phi(v^\prime,v,c)M^\prime \d v^\prime} 
     = \eps^{\alpha-1}{\cal F}[G_\eps](v) + \frac{1}{1+\eps^{\alpha-1}\int \Phi(v^\prime,v,c)M^\prime \d v^\prime}  \,,
$$
which implies the estimates \eqref{F-eps-est} in a straightforward way (using the normalization of $M$ and $F_\eps$).

For the time derivative, we get
$$
  \partial_t G_\eps = \eps^{\alpha-1} \left( {\cal F}[\partial_t G_\eps] + \frac{\left( G_\eps\int \nabla_c \Phi(v^\prime,v,c)M^\prime 
  \d v^\prime + \int \nabla_c \Phi(v,v^\prime,c)M^\prime G_\eps^\prime \d v^\prime\right)\cdot\partial_t c}
  {1+\eps^{\alpha-1}\int \Phi(v^\prime,v,c)M^\prime \d v^\prime} \right) \,,
$$
which is of course only relevant in the case of Assumption C in Theorem \ref{Theorem1}. As a consequence of this assumption 
(boundedness of $\nabla_c \Phi$ and of $\partial_t c$) and
of the uniform $L^\infty$ bound \eqref{F-eps-est} for $G_\eps$, the inhomogeneity is $O(\eps^{\alpha-1})$, uniformly in $(x,t,v)$.
This implies, again by contraction, a uniform $L^\infty$ bound of $O(\eps^{\alpha-1})$ for $\partial_t G_\eps$, giving the bound on
$\partial_t F_\eps/F_\eps$ in \eqref{F-eps-deriv}, again as a consequence of \eqref{F-eps-est}.

Analogously, the components of $\nabla_x G_\eps$ are shown to be $O(\eps^{\alpha-1})$. Finally, multiplication of the equation
for $\nabla_x G_\eps$ by $v$ and using the boundedness of $(|v| + |v^\prime|)\Phi$ and $(|v| + |v^\prime|)\nabla_c\Phi$
leads to an $O(\eps^{\alpha-1})$ bound on $v\cdot\nabla_x G_\eps$ and therefore also for $v\cdot\nabla_x F_\eps/F_\eps$.
\end{proof}

\begin{remark}
As consequences of \eqref{F-eps-est},
\begin{itemize}
 \item[ \emph{ (a)} ] $\displaystyle \mu_1 M \leq F_\varepsilon \leq \mu_2 M$,
 \item[ \emph{ (b)} ] $\displaystyle | F_\varepsilon - M | \leq  \varepsilon^{ \alpha - 1} \mu_3 M$,
\end{itemize}
hold with $\eps$-independent constants $\mu_1,\mu_2,\mu_3$, which has already been used in the above proof and will be used
in the following. 
\end{remark}

Entropy decay properties for collision operators with detailed balance have been a classical tool in kinetic theory. 
The detailed balance assumption has been dispensed with in \cite{MR1803225} (see also \cite{FontbonaJourdain}), where a proof of 
the following result can be found.

\begin{lemma}\label{coercivitylemma}
Let the assumptions of Theorem \ref{Theorem1} hold and let $\eps$ be small enough such that $1+\eps^{\alpha-1}\Phi \ge \nu\mu_2>0$.
Then the collision operator $Q_\eps$ satisfies the coercivity inequality
\begin{equation}\label{coercivityIneq}
- \int Q_\eps(f) \frac{f}{F_\eps} \d v \geq \nu
  \| f - \rho_f F_\eps \|^2_{L^2 (\d v/F_\eps)} \qquad \text{ for all } f \in L^2(\d v/F_\eps) \,.
\end{equation}
\end{lemma}

The existence and uniqueness of a nonnegative solution $f_\eps$ of \eqref{kineticeq} for small enough $\eps$ is a classical result 
of kinetic theory and will be assumed here. The coercivity result Lemma \ref{coercivitylemma} will be used for the derivation of
bounds for the solution. The dependence of the equilibrium distribution $F_\eps$ on $x$ and $t$ destroys entropy
decay, but fortunately uniform bounds on finite time intervals will still be possible.

Lemma \ref{coercivitylemma} suggests the use of $L^2$-norms with weight $1/F_\eps$:
\[
  \frac{\eps^\alpha}{2} \frac{\d}{\d t} \|f_\eps\|^2_{L^2(dv\d x/F_\eps)} &=& - \frac{\eps^\alpha}{2}\int\int 
    \frac{\partial_t F_\eps}{F_\eps} \frac{f_\eps^2}{F_\eps}\d v\d x + \frac{\eps}{2}\int\int
    \frac{v\cdot\nabla_x F_\eps}{F_\eps} \frac{f_\eps^2}{F_\eps}\d v\d x \\
&& + \int\int Q_\eps(f_\eps)\frac{f_\eps}{F_\eps} \d v\d x \,,
\]
where the second term on the right hand side is the result of an integration by parts. Now we use \eqref{F-eps-deriv} and
\eqref{coercivityIneq}:
\(\label{L2-est}
  \frac{\eps^\alpha}{2} \frac{\d}{\d t} \|f_\eps\|^2_{L^2(\d v\d x/F_\eps)} \le  \eps^\alpha \lambda \|f_\eps\|^2_{L^2(\d v\d x/F_\eps)} 
  - \nu \| f_\eps - \rho_\eps F_\eps \|^2_{L^2 (\d v\d x/F_\eps)}\,,
\)

\begin{theorem}\label{apriorieq1}
Let the assumptions of Theorem \ref{Theorem1} hold. Then for small enough $\eps>0$
\begin{itemize}
 \item[ \emph{(i)} ] $f_\eps$ is uniformly (with respect to $\eps$) bounded in $L^\infty \left( \d t; L^1(\d v\d x)\right)$ and in \\ $L^\infty\left( e^{-\lambda t}\d t; L^2 (\d v \d x/M)\right)$ with $\lambda$ from \eqref{L2-est}, vanishing under Assumptions A and B,
 \item[ \emph{(ii)} ]$\rho_\eps$ is uniformly bounded in $L^\infty \left( e^{-\lambda t}\d t; L^2 (\d x)\right)$ and 
 in $L^\infty \left( \d t; L^1(\d x)\right)$,
 \item[ \emph{(iii)} ] $r_\eps := \eps^{1-\alpha}(f_\eps - \rho_\eps M)$ is uniformly bounded in 
 $L^2\left(e^{-2\lambda t}\d t; L^2( \d v\d x/M)\right)$.
\end{itemize}
\end{theorem}

\begin{remark}
By uniform boundedness of $f_\eps$ in $L^\infty\left( e^{-\lambda t}\d t; L^2 (\d v \d x/M)\right)$, we mean that
$e^{-\lambda t}\|f_\eps\|_{L^2 (\d v \d x/M)}$ is bounded uniformly in $\eps$ and in $t\in\R^+$.
\end{remark}

\begin{proof}
The first result (i) is a consequence of the Gronwall lemma, after neglecting the last term in \eqref{L2-est} and of the conservation of
total mass. Note that by Remark 1 (a) the weights $1/M$ and $1/F_\eps$ are equivalent. Then (ii) follows from the inequality
$$
  \rho_\eps \le \|f_\eps\|_{L^2(dv/M)} \,,
$$
derived from the Cauchy-Schwarz inequality and the normalization of $M$. Finally, (iii) is a consequence of \eqref{L2-est}
after integration with respect to $t$, using
$$
  |f_\eps - \rho_\eps M| \le |f_\eps - \rho_\eps F_\eps| + \rho_\eps |F_\eps - M| \le |f_\eps - \rho_\eps F_\eps| 
  + \eps^{\alpha-1} \mu_3 \rho_\eps M \,,
$$
by Remark 1 (b). Note that $\eps^{\alpha/2} < \eps^{\alpha-1}$.
\end{proof}

Since $M$ has moments of any order smaller than $\alpha$, existence of these moments is propagated by the kinetic equation.
We shall need the first order moment:

\begin{lemma}\label{momentGe}
Let the assumptions of Theorem \ref{Theorem1} and Proposition 1 be satisfied and let $\eps$ be small enough. Then 
$\int\int |v|f_\eps \d v\d x$ is bounded uniformly with respect to $t$ and $\eps$.
\end{lemma}

\begin{proof}
Since for $\eps$ small enough, $1+\eps^{\alpha-1}\Phi$ is uniformly bounded from above and away from zero, after multiplication of
\eqref{kineticeq} by $| v|$ and integration with respect to $x$ and $v$, we estimate
\[
  \eps^\alpha \frac{\d}{\d t} \int\int |v| f_\eps \d v \d x  \le C_1 - C_2 \int\int |v| f_\eps \d v \d x \,,
\]
implying 
\[
  \int\int |v| f_\eps \d v \d x \le \max\left\{ \frac{C_1}{C_2}, \int\int |v| f^{in} \d v \d x \right\} \,.
\]
\end{proof}

\section{Rigorous asymptotics for constant $c$}

In this section, we shall prove Theorem \ref{Theorem1} under Assumption A or B, following the strategy of \cite{MR2763032}.  
Analogously to the derivation of \eqref{rho_eps} in Section \ref{sec:formal}, Fourier-Laplace transformation of (\ref{kineticeq}) yields
\(
  &&\widehat{\rho_{\varepsilon}} \int \frac{p + \eps^{\alpha}p^2 + \eps^{2-\alpha} (v\cdot k)^2}
  {(1+\eps^\alpha p)^2 + \eps^2 (v\cdot k)^2}M\d v 
  - \int \frac{\mathcal{F} f^{in}}{1 + \varepsilon^\alpha p + \varepsilon i v \cdot k} \d v  \nonumber\\
  && = \frac{1}{\eps}
  \int \frac{Q_1\left(\widehat{ f_\eps}\right)}{1 + \varepsilon^\alpha p + \varepsilon i v \cdot k}\d v \,.\label{rho_eps1}
\)
The rigorous passage to the limit in the left hand side of this equation has already been carried out in \cite{MR2763032}. 
For completeness, we repeat the essential arguments and start with the second term, whose convergence as $\eps\to 0$ to 
$\mathcal{F} \rho^{in}$ follows from the dominated convergence theorem, noting 
$|1 + \varepsilon^\alpha p + \varepsilon i v \cdot k|\ge 1$ and
$$
   |\mathcal{F} f^{in}| \le \int f^{in} \d x \in L^1(\d v) \,.
$$
The dominated convergence theorem also implies
$$
   \lim_{\eps\to 0} \int \frac{p + \eps^{\alpha}p^2 }
  {(1+\eps^\alpha p)^2 + \eps^2 (v\cdot k)^2}M\d v = p \,,\qquad\forall\, p>0\,,\ k\in\R^N \,.
$$
Furthermore we have
$$
   \int_{|v|<1} \frac{\eps^{2-\alpha} (v\cdot k)^2}
  {(1+\eps^\alpha p)^2 + \eps^2 (v\cdot k)^2}M\d v \le \eps^{2-\alpha} |k|^2 \to 0 \quad\mbox{as } \eps\to 0 \,.
$$
In the integral over $|v|>1$ we use \eqref{eqdecay} and carry out the coordinate transformation $v = (w_1 k/|k| + w^\bot)/(\eps |k|)$ 
(a stretching and a rotation): 
$$
   \int_{|v|>1} \frac{\eps^{2-\alpha} (v\cdot k)^2}
  {(1+\eps^\alpha p)^2 + \eps^2 (v\cdot k)^2}M\d v 
  = \gamma |k|^\alpha \int_{|w|>\eps|k|} \frac{w_1^2 |w|^{-N-\alpha}}{(1+\eps^\alpha p)^2+w_1^2} \d w
$$
Again by dominated convergence, the right hand side converges as $\eps\to 0$ for all $p>0$, $k\in\R^N$, to 
$$
  A |k|^\alpha \,,\qquad\mbox{with } A = \gamma\int_{\R^N} \frac{w_1^2 |w|^{-N-\alpha}}{1+w_1^2} \d w > 0 \,.
$$
As a consequence of Theorem \ref{apriorieq1}, there exists $\rho\in L^\infty \left( \d t; L^2(dx) \right) \cap
L^\infty \left( \d t; L^1(dx)  \right)$,
such that 
$$
  \rho_\eps\rightharpoonup\rho \qquad\mbox{in } L^\infty \left( \d t; L^2(dx) \right) \mbox{ weak*.}
$$
Since 
$$
  \left| \widehat{\rho_\eps}(k,p)\right| \le \frac{1}{p} \|\rho_\eps\|_{L^\infty \left( \d t; L^1(dx)  \right)} \,,
$$
$\widehat{\rho_\eps}$ is uniformly bounded in $L^\infty\left((a,\infty)\times\R^N\right)$ for $a>0$, implying
\begin{equation} \label{hatr-conv}
  \widehat{\rho_\eps} \rightharpoonup \widehat\rho \qquad\mbox{in } L^\infty\left((a,\infty)\times\R^N\right) \mbox{ weak*.}
\end{equation}
Our results so far imply distributional convergence of the left hand side of \eqref{rho_eps1} to 
$$
  (p+A|k|^\alpha)\widehat\rho - \mathcal{F} \rho^{in} \,.
$$
Now we turn to the right hand side and observe that due to mass conservation $\left(\int Q_1(f)\d v = 0\right)$
and with the notation of Theorem \ref{apriorieq1} (iii)
\begin{eqnarray}
  &&\frac{1}{\eps} \int \frac{Q_1\left(\widehat{ f_\eps}\right)}{1 + \varepsilon^\alpha p + \varepsilon i v \cdot k}\d v
   = - \frac{ik}{1+\eps^\alpha p} 
   \cdot \int \frac{v Q_1\left(\widehat{ f_\eps}\right)}{1 + \varepsilon^\alpha p + \varepsilon i v \cdot k}\d v \nonumber\\
   &&= - \frac{ik}{1+\eps^\alpha p} 
   \cdot \left(\widehat{\rho_\eps}\int \frac{v Q_1(M)}{1 + \varepsilon^\alpha p + \varepsilon i v \cdot k}\d v +
   \eps^{\alpha-1}\int \frac{v Q_1\left(\widehat{ r_\eps}\right)}{1 + \varepsilon^\alpha p + \varepsilon i v \cdot k}\d v\right)
   \label{Q1-term}
\end{eqnarray}
holds. With $\overline\Phi := \sup_{v,v^\prime,c} \Phi$, it is straightforward to show $|Q_1(M)| \le \overline\Phi M$. Since the first
order moments of $M$ are finite, the dominated convergence theorem implies
$$
  \lim_{\eps\to 0} \int \frac{v Q_1(M)}{1 + \varepsilon^\alpha p + \varepsilon i v \cdot k}\d v = \int v Q_1(M) \d v = u(c)\,.
$$
For the last integral in \eqref{Q1-term}, we start with the case of Assumption B (satisfied by the example treated in Section \ref{sec:formal}), whence $Q_1(f) = M\int \Phi(v,v^\prime,c)f^\prime \d v^\prime$. This implies
\begin{equation} \label{Q1-est}
  |Q_1(\widehat{r_\eps})| \le \overline\Phi M  \int |{\widehat{r_\eps}}^\prime| \d v^\prime \le
  \overline\Phi M \|\widehat{r_\eps} \|_{L^2(\d v/M)}
\end{equation}
such that
$$
  \left| \int \frac{v Q_1\left(\widehat{ r_\eps}\right)}{1 + \varepsilon^\alpha p + \varepsilon i v \cdot k}\d v \right| \le 
  \overline\Phi \int |v|M\d v\, \|\widehat{r_\eps} \|_{L^2(\d v/M)}
$$
is, by Theorem \ref{apriorieq1}, uniformly bounded in $L^\infty((a,\infty); L^2(\d k))$ by the estimate (using the Cauchy-Schwarz
inequality and the Plancherel identity)
\begin{align}
  \| \widehat{r_\varepsilon}\|^2_{L^\infty ( ( a, \infty); L^2 ( \d k \d v/M))} 
  &= \sup_{ p \geq a} \int \int \bigg| \int_0^\infty \e^{ -p t} \mathcal{F} r_\varepsilon \d t \bigg|^2 \frac{ \d k \d v}{ M}  \nonumber \\
  & \leq \sup_{ p \geq a} \int \int \left( \int_0^\infty \e^{ - 2pt} \d t \right) 
      \left( \int_0^\infty | \mathcal{F} r_\varepsilon|^2 \d t \right) \frac{ \d k \d v}{ M} \nonumber \\
  & = \frac{1}{2 a} \|\mathcal{F} r_\varepsilon\|^2_{L^2(\d t\d k\d v/M)} = C\|r_\varepsilon\|^2_{L^2(\d t\d x\d v/M)} \,.\label{hatr-est}
\end{align}
In the case of Assumption A, i.e. $\sup_{v,v^\prime,c}(|v| + |v^\prime|)\Phi(v,v^\prime,c) =: \overline\Phi_1 <\infty$, we estimate
$$
  |vQ_1(\widehat{r_\eps})| \le \overline\Phi_1 \left( M\int |\widehat{r_\eps}^\prime| \d v^\prime + |\widehat{r_\eps}| \right)
$$
which, by \eqref{hatr-est} and by an estimate like in \eqref{Q1-est}, is uniformly bounded in $L^\infty ( ( a, \infty); L^2 ( \d k \d v/M))$.
Thus, under both Assumptions A and B, the last term in \eqref{Q1-term} is $O(\eps^{\alpha-1})$ in 
$L^\infty ( ( a, \infty); L^2 ( \d k \d v/M))$.

Finally, using again \eqref{hatr-conv}, we can pass to the limit also in the right hand side of \eqref{rho_eps1} and obtain
$$
(p+ ik\cdot u(c)+A|k|^\alpha )\widehat\rho = \mathcal{F} \rho^{in} \,,
$$
the Fourier-Laplace transform of \eqref{macro}, concluding the proof.

\section{Rigorous asymptotics for non-constant $c$}

In this section we use a completely different approach, introduced by Mellet \cite{MR2815035} and called the \emph{moment method}. 
For a test function $\phi(x,t) \in C^\infty_0( \R^N \times [0,\infty))$, we denote by $\chi_\varepsilon(x,v,t)$ the unique bounded solution of the auxiliary equation
\begin{equation}\label{auxiliaryeq1}
\chi_\varepsilon - \varepsilon v \cdot \nabla_x \chi_\varepsilon = \phi \,,
\end{equation}
which can be computed explicitly via the method of characteristics:
\begin{equation}\label{explicitaux}
\chi_\varepsilon ( x, v, t) = \int_0^\infty \e^{-z} \phi ( x + \varepsilon v z, t) \d z \,. 
\end{equation}
The operator on the left hand side of \eqref{auxiliaryeq1} is the adjoint of a part of the operator appearing in \eqref{kineticeq},
consisting only of the transport operator and of the loss term of the leading order collision operator $Q_0$.
Some properties of $\chi_\eps$ are collected in the following lemma, mostly proven already in \cite{MR2815035}.

\begin{lemma}\label{unifconveq1}
Let $\varphi \in \D ( \R^N \times [0,\infty))$, and let $\chi_\eps$ be defined by \eqref{explicitaux}. Then
$\chi_\eps$, $\partial_t\chi_\eps$, and $\nabla_x\chi_\eps$ are bounded in $L^\infty(\d v\d x\d t)$ and in $L^2(M\d v\d x\d t)$
uniformly in $\eps$. Furthermore
$$
  \int M|\chi_\eps-\phi|\d v\,,\, \int M|\partial_t\chi_\eps-\partial_t\phi|\d v\,,\, \int M|\nabla_x\chi_\eps-\nabla_x\phi|\d v 
  = O(\eps) \,,
$$
uniformly in $x$ and $t$.
\end{lemma}

\begin{proof}
The boundedness statements in $L^\infty$ are a straightforward consequence of the boundedness of $\phi$ and of its derivatives. 
Because of $\int_0^\infty e^{-z}\d z = 1$,
$$
  \chi_\eps(x,v,t)^2 \le \int_0^\infty e^{-z} \phi(x+\eps vz,t)^2 \d z 
$$
holds, and therefore
$$
  \|\chi_\eps\|_{L^2(M\d v\d x\d t)}^2 \le \int_0^\infty \int\int \int_0^\infty M(v)e^{-z}\phi(x+\eps vz,t)^2 \d z\d v\d x\d t
  = \|\phi\|_{L^2(\d x\d t)}^2 \,,
$$
and the same argument for the derivatives.

On the other hand, with the Lipschitz constant $L$ of $\phi$,
\begin{eqnarray}
  \int M|\chi_\eps-\phi|\d v  &=& \int M \left|\int_0^\infty e^{-z}(\phi(x+\eps vz,t) -\phi(x,t))\d z\right| \d v \nonumber\\
  &\le&  \eps L \int |v|M\d v \int_0^\infty ze^{-z}\d z \,,\label{chi2phi}
\end{eqnarray}
implying the desired result by the finiteness of the first order moments of $M$. The proof of the remaining two statements is analogous.
\end{proof}

Multiplication of the kinetic equation (\ref{kineticeq}) by $\chi_\eps$, integration with respect to $x$, $v$, and $t$, and using
\eqref{auxiliaryeq1}, gives
\begin{align}
&  \int_0^\infty \int\int f_\eps \partial_t \chi_\eps \d v\d x\d t + \int\int f^{ in} \chi_\eps (t=0) \d v\d x 
 + \eps^{-\alpha} \int_0^\infty \int\int \rho_\eps M (\chi_\eps - \phi)\d v\d x\d t  \nonumber \\ 
& = - \frac{1}{\eps} \int_0^\infty \int\int Q_1(f_\eps)\chi_\eps\d v\d x\d t \label{basicGe}
\end{align} 
The rest of the proof is concerned with the passage to the limit $\eps\to 0$ in each of the terms of \eqref{basicGe}.  
Similarly to the preceding section, we outline the arguments for the terms on the left hand side, which have already been
treated in \cite{MR2815035}.

Rewriting the first term on the left hand side leads to 
$$
    \int_0^\infty \int\int f_\eps \partial_t \phi \d v\d x\d t 
    + \int_0^\infty \int\int f_\eps (\partial_t \chi_\eps -\partial_t \phi)\d v\d x\d t \to
    \int_0^\infty \int\rho\, \partial_t \phi \d x\d t \,,
$$
as $\eps\to0$, where $f_\eps \to \rho M$ in the sense of distibutions as a consequence of Theorem \ref{apriorieq1}.
The second term above vanishes in the limit by an argument analogously to \eqref{chi2phi}, since, by Lemma \ref{momentGe},
$f_\eps$ has first order moments in $v$, integrable with respect to $x$ and bounded in $t$. In the same way
$$
  \lim_{\eps\to 0} \int\int f^{ in} \chi_\eps (t=0) \d v\d x = \int\rho^{ in} \phi(t=0) \d x
$$
is proven, using the integrability with respect to $x$ of the first $v$-moments of $f^{in}$, as assumed in Theorem \ref{Theorem1}.

The third term in \eqref{basicGe} leads to the fractional diffusion operator. By the rotational symmetry of $M$, we have
\begin{equation*}
    \eps^{-\alpha} \int M (\chi_\eps - \phi)\d v 
    = \eps^{-\alpha} \int \int_0^\infty M e^{-z}(\phi(x+\eps vz) - \phi(x) - \eps vz\cdot\nabla_x\phi(x)) \d z\d v \,.
\end{equation*}
This implies
$$
  \left|\eps^{-\alpha} \int_{|v|<1} M (\chi_\eps - \phi)\d v\right| 
  \le \eps^{2-\alpha} C \int_{|v|<1} |v|^2M\d v \int_0^\infty z^2e^{-z}\d z \,.
$$
In the integral over $|v|>1$, we introduce the coordinate transformation $v\leftrightarrow w=\eps vz$ to obtain
\begin{eqnarray*}
    &&\eps^{-\alpha} \int_{|v|>1} M (\chi_\eps - \phi)\d v  
 = \gamma\int_0^\infty\int_{|w|>\eps z}  z^\alpha e^{-z} \frac{\phi(x+w) - \phi(x) - w\cdot\nabla_x\phi(x)}{|w|^{N+\alpha}} \d w\d z \\  
 && \to -A(-\Delta)^{\alpha/2} \phi \,.
\end{eqnarray*}
The limit is uniform in $x$ and $t$, due to the estimate
\begin{eqnarray*}
  &&\left|\int_0^\infty\int_{|w|<\eps z}  z^\alpha e^{-z} \frac{\phi(x+w) - \phi(x) - w\cdot\nabla_x\phi(x)}{|w|^{N+\alpha}} 
    \d w\d z\right| \\
  &&\le C \int_0^\infty\int_{|w|<\eps z} z^\alpha e^{-z} |w|^{2-N-\alpha}\d w\d z \le \eps^{2-\alpha}C \int_0^\infty z^2 e^{-z}\d z \,.
\end{eqnarray*}
As a consequence, uniform integrability and weak convergence of $\rho_\eps$ are sufficient for passing to the limit in the
third term of \eqref{basicGe}. Collecting our results so far, the left hand side of \eqref{basicGe} converges to 
$$
  \int_0^\infty \int\rho\left( \partial_t \phi - (-\Delta)^{\alpha/2} \phi\right)\d x\d t + \int\rho^{ in} \phi(t=0) \d x \,.
$$
Finally we consider the right hand side of \eqref{basicGe}, and use the mass conservation property of $Q_1$, the properties of
$\chi_\eps$, and the macro-micro decomposition of $f_\eps$:
\begin{eqnarray}
  &&\frac{1}{\eps} \int Q_1(f_\eps)\chi_\eps\d v = \int Q_1(f_\eps)v\cdot\nabla_x\chi_\eps\d v \nonumber\\
  &&= \rho_\eps u(c)\cdot\nabla_x\phi
  + \rho_\eps \int Q_1(M)v\cdot(\nabla_x \chi_\eps - \nabla_x\phi)\d v + \eps^{\alpha-1}\int Q_1(r_\eps)v\cdot\nabla_x\chi_\eps \d v
  \label{Q1-est2}
\end{eqnarray}
After integration with respect to $x$ and $t$, we can pass to the limit in the first term on the right hand side by the weak
convergence of $\rho_\eps$. By Assumption C we can use again
$\sup_{v,v^\prime,c}(|v| + |v^\prime|)\Phi(v,v^\prime,c) = \overline\Phi_1 <\infty$ to obtain 
\begin{equation}\label{Q1-est3}
  |Q_1(f)v| \le \overline\Phi_1 \left( M\int |f^\prime|\d v^\prime + |f|\right) \,.
\end{equation}
In particular, the consequence $|Q_1(M)v| \le 2\overline\Phi_1 M$ implies by Lemma \ref{unifconveq1} that the integral in the
second term on the right hand side of \eqref{Q1-est2} is $O(\eps)$ uniformly in $x$ and $t$. For the last term in \eqref{Q1-est2}
we have, after integration with respect to $x$ and $t$, by \eqref{Q1-est3}, by the Cauchy-Schwarz inequality, and by Lemma \ref{unifconveq1}
\begin{eqnarray*}
  &&\left|\int_0^\infty \int\int Q_1(r_\eps)v\cdot\nabla_x\chi_\eps \d v\d x\d t\right| \\
  &&\le C \left(\left\| M\int |r_\eps^\prime|\d v^\prime \right\|_{L^2((0,T); L^2(\d v\d x/M))} + \| r_\eps\|_{L^2((0,T); L^2(\d v\d x/M))}\right) \\
  &&\le 2C \| r_\eps\|_{L^2((0,T); L^2(\d v\d x/M))} \,,
\end{eqnarray*}
where $T<\infty$ denotes an upper bound for $t$ in the support of $\phi$. The right hand side is uniformly bounded by 
Theorem \ref{apriorieq1}.
Combining our results, the limit of \eqref{basicGe} as $\eps\to 0$ reads
$$
  \int_0^\infty \int\rho\left( \partial_t \phi + u(c)\cdot\nabla_x\phi - (-\Delta)^{\alpha/2} \phi\right)\d x\d t + \int\rho^{ in} \phi(t=0) \d x = 0\,,
$$
which is the distributional formulation of \eqref{macro}.

\bibliographystyle{siam}

\bibliography{bibliography.bib}

\end{document}